\documentclass[10pt]{article}
\usepackage[usenames]{color}
\usepackage{xcolor}
\usepackage{hyperref}
\usepackage[utf8]{inputenc}
\usepackage{amsmath,amsfonts,amscd,amssymb,amsthm,latexsym,eucal,graphics,amsbsy,romannum}
\providecommand{\N}{{\mathrm{N}}} 
\providecommand{\C}{{\mathrm{C}}}
\providecommand{\Z}{{\mathrm{Z}}}

\providecommand{\Ind}{{\operatorname{Ind}}}
 
\providecommand{\gen}[1]{\langle #1 \rangle}

\providecommand{\gen}[1]{\langle #1 \rangle}

\newtheorem{lem}{Lemma}
\newtheorem*{theorem}{Main Theorem}

\newtheorem*{cor}{Corollary}

\newtheorem{question}{Question}

\title{On Groups with the Same Set of Conjugacy Class Sizes as Nilpotent Groups}
\author{
    {Zhou Wei}\\
    {\small Sobolev Institute of Mathematics, Novosibirsk 630090, Russia}\\
    {\small Email: zhouwyx@outlook.com}
}

\date{}
\begin{document}
\pagenumbering{arabic}

\maketitle
\noindent\textbf{Abstract.} 
    We construct examples of groups 
    which have the same set of conjugacy class sizes as nilpotent groups, 
    while their center is trivial. 
    This answers a question posed by A. R. Camina in 2006.



%***************************************************************************
\section{Introduction}
    Let $G$ be a finite group.
    In \cite{Baer1953}, Baer defined the index of $x$ in $G$, 
    denoted by $\Ind_G(x)$, as $|G : \C_G(x)|$, 
    which represents the size of the conjugacy class of $G$ containing $x$.
    In $\cite{Ito1953}$, 
    It\^o defined the conjugate type vector of $G$ as $(n_1, n_2, \ldots, n_r)$, 
    where $n_1 > n_2 > \ldots > n_r = 1$ are the indices of all elements in $G$.
    Since we are not interested in the ordering of these indices, 
    we will denote the set of indices (sizes of conjugacy classes) by $\N(G)$,
    i.e., $\N(G)=\{n_1,n_2,\ldots,n_r\}$.
    
    Many authors have studied the relationship 
    between the structure of finite groups and the sizes of their conjugacy classes.
    It\^o proved that 
    if $\N(G) = \{1, n\}$, 
    then $G$ must be the direct product of a $p$-group and an abelian $p'$-group \cite{Ito1953}.
    Ishikawa proved that the nilpotent class of such groups is at most $3$ \cite{Ishikawa2002}.
    More results can be found in \cite{Camina2011}. 

    It is easy to see that if $G$ is nilpotent, 
    then $\N(G) = \N(P_1) \times \N(P_2) \times \ldots \times \N(P_k)$, 
    where $P_1,P_2, \ldots, P_k$ are the Sylow subgroups of $G$.
    A natural question is whether the converse holds:

\begin{question}[{\cite[Question 1]{Camina2006}}]
\label{Q1}
    Let $G$ and $H$ be finite groups with $H$ nilpotent. 
    Suppose $G$ and $H$ have the same sets of conjugacy class sizes, 
    is $G$ nilpotent?
\end{question}

    In \cite{Cossey2000}, Cossey proved that 
    every finite set of $p$-powers containing $1$
    can be the set of conjugacy class sizes of some $p$-group.
    Therefore, the above question can be restated as follows:
    If $\N(G)=\Omega_1\times \Omega_2 \times \cdots \times \Omega_r$,
    where $\Omega_i$ is a finite set of $p_i$-powers containing $1$
    and $p_1,p_2,\ldots,p_r$ are distinct primes,
    is $G$ nilpotent?
    The answer is positive in some special cases.
    For example,  
    % if all the conjugacy class lengths are square-free, 
    % then the converse holds \cite{Camina1998}. 
    if $\N(G)=\{1,p_1^{m_1}\}\times \{1,p_2^{m_2}\}\times \cdots \times \{1,p_k^{m_k}\}$,
    where $p_1^{m_1},p_2^{m_2},\ldots,p_k^{m_k}$ are powers of distinct primes,
    then $G$ is nilpotent \cite{Casolo2012}.
    More generally, 
    if $\N(G) = \{1, n_1\} \times \{1, n_2\} \times \cdots \times \{1, n_r\}$, 
    where $n_1, n_2, \ldots, n_r$ are pairwise coprime integers, 
    then $G$ is nilpotent \cite{Gorshkov2024}.
    A more general question is as follows:

    \begin{question} [{\cite[Question 0.1]{Gorshkov2023}}]
        Let $G$ be a group such that $\N(G)=\Omega \times \Delta$.
        Which $\Delta$ and $\Omega$ guarantee that $G\cong A\times B$,
        where $A$ and $B$ are subgroups 
        such that $\N(A)=\Omega$ and $\N(B)=\Delta$?
    \end{question}

    However, the answer to Question \ref{Q1} is not always true, 
    as some counterexamples are provided in \cite{Camina2006}.
    In that paper, 
    A. R. Camina posed a number of questions 
    about the structure of groups 
    with the same set of conjugacy class sizes as nilpotent groups.
    One of them is as follows:

\begin{question} [{\cite[Question 4]{Camina2006}}]
\label{Q3}
    Let $G$ and $H$ be finite groups with $H$ nilpotent. 
    Suppose $\N(G)=\N(H)$, 
    but $G$ is not nilpotent. 
    Does $G$ have a nontrivial centre?
\end{question}
    
    Using GAP\cite{GAP}, 
    we find that Question \ref{Q3} does not have a positive answer in general.
    The smallest counterexamples are two groups of order $486=3^5\times 2$,
    with the set of conjugacy class sizes $\{1,3,27\}\times \{1, 2\}$.
    One of them is SmallGroup(486, 36), and the other is SmallGroup(486, 38).
    Moreover, we constructed the following series of counterexamples.
    
\begin{theorem}
    Let $p$ and $q$ be primes such that $p = 2q + 1$.
    Let $G = H \rtimes (A \rtimes B)$, 
    where $H$, $A$ and $B$ are defined as follows:

    1) $H = K/N$, 
    where $ K = \gen{k_1} \times \gen{k_2} \times \ldots \times \gen{k_p}$ 
    is the direct product of $ p $ cyclic groups of order $p$, 
    and $N = \gen{k_1 k_2 \dots k_p}$;

    2) $A\rtimes B$ is a subgroup of the symmetric group $Sym_p$:
    $A=\gen{\alpha}$ and 
    $B=\gen{\beta}$, 
    where $\alpha=(12\ldots p)$ and $\beta = (m_1\ldots m_q)(n_1\ldots n_q)$,
    with $\{m_1,\ldots,m_q,n_1,\ldots,n_q\}=\{2,3,\ldots,p\}$.
    Additionally, 
    $\alpha ^{\beta}=\alpha ^r$
    where $1<r<q$ and $r^q \equiv 1 \pmod{p}$.
    For any $\gamma\in A\rtimes B$ 
    and $k_1^{x_1}k_2^{x_2}\ldots k_p^{x_p}N\in H$,
    $(k_1^{x_1}k_2^{x_2}\ldots k_p^{x_p}N)^{\gamma}
    =k_{1^{\gamma}}^{x_1}k_{2^{\gamma}}^{x_2}\ldots k_{p^{\gamma}}^{x_p}N$.

    Then $\N(G)=\{1, p, p^{p-2}\}\times \{1, q\}$, 
    and $\Z(G)=1$.
\end{theorem}

    From this theorem, the following corollary can be derived.

\begin{cor}
    Let $p$ and $q$ be primes such that $p = 2q + 1$. 
    Let $G$, $H$ and $A$ be as defined above. 
    Let $L = P \times Q$, 
    where $P = H \rtimes A$ and $Q = C_{q^2}\rtimes C_q$. 
    Then, we have $\N(G) = \N(L)$.
\end{cor}

    A prime number $q$ such that $2q + 1$ is also a prime is called a Sophie Germain prime.
    The largest known proven Sophie Germain prime is $2618163402417\times 2^{1290000}-1$ \cite{Bastan2018}.
    It is conjectured that there are infinitely many Sophie Germain primes, 
    but this has not been proven.
    So we cannot conclude that there are infinitely many counterexamples to Question \ref{Q3}.

%***************************************************************************
\section{Preliminaries}

\begin{lem}
\label{L1}
    Let $G$ be a finite group, $H\leq G$ and $x\in G$. 
    If $n$ is an integer and $(n,|x|)=1$, then $\C_H(x)=\C_H(x^n)$.
\end{lem}

\begin{proof}
    It is clear that $\C_H(x)\leq \C_H(x^n)$.
    By Euler's theorem, we have $n^t\equiv 1 \pmod{|x|}$, where $t=\varphi(|x|)$.
    Hence $x=(x^n)^{n^{t-1}}$ and so $\C_H(x^n)\leq \C_H(x)$.
    Therefore $\C_H(x)=\C_H(x^n)$.
\end{proof}

\begin{lem}
\label{L2}
    Let $G=H\rtimes \gen{a}$, where $H$ is an abelian group and $(|H|,|a|)=1$.
    Then for any element $h$ of $H$, $\Ind_G(ha)=\Ind_G(a)=|H:\C_H(a)|$.
\end{lem}

\begin{proof}
    Let $|a|=n$. 
    It is easy to verify that $\C_G(a)=\C_H(a)\gen{a}$ and $\Ind_G(a)=|H:\C_H(a)|$.
    Since 
    $(ha)^n = h h^{a^{-1}} \ldots h^{a^{1-n}} a^n = h h^{a^{-1}} \ldots h^{a^{1-n}} \in H$,
    $n$ is a divisor of $|ha|$.
    Let $t=|ha|/n$. 
    We have $(ha)^t$ is an element of order $n$ 
    and $\gen{(ha)^t}$ is a complement to $H$ in $G$.
    Hence $G=H\rtimes \gen{(ha)^t}$ 
    and so $\C_G((ha)^t)=\C_H((ha)^t)\gen{(ha)^t}$.
    Since $(|H|,n)=1$, we have $(t, n)=1$.
    By Lemma \ref{L1}, $\C_H((ha)^t)=\C_H(a^t)=\C_H(a)$.
    Since $\C_H(a)\gen{(ha)^t}\leq \C_G(ha)\leq \C_G((ha)^t)$,
    we have $\C_G(ha)=\C_H(a)\gen{(ha)^t}$.
    Therefore $\Ind_G(ha)=\Ind_G(a)=|H:\C_H(a)|$.
\end{proof}

%***************************************************************************
\section{Proof the main theorem}
    Let $G, A, B, H, N$ be as defined in the main theorem.
    For convenience,
    we use $(x_1, x_2, \ldots, x_p)$ to represent the element 
    $k_1^{x_1} k_2^{x_2} \ldots k_p^{x_p} N$ of $H$, $x_1, \ldots, x_p \in \mathbb{N}$.
    Under this notation, we have $(x,x,\ldots,x)=1$, $\forall x\in \mathbb{N}$.
    We can always set $x_1 = 0$, 
    in which case $x_2, \ldots, x_p$ are determined.
    Let $h = (0, x_2, \ldots, x_p) \in H$, $a\in A$, $b\in B$ and $h, a, b\neq 1$.
    It is clear that $|G|=p^pq$.

(1) $|\C_H(a)|=p$ and $\Ind_G(a)=p^{p-2}q$.

    By Lemma \ref{L1}, it suffices to consider the case $a = \alpha$,
    i.e., when $(0, x_2, \ldots, x_{p-1}, x_p)^a = (x_p, 0,x_1, \ldots, x_{p-1})$.
    If $h \in \C_H(a)$, we have
    $ 0 - x_p\equiv  x_2 - 0 \equiv \ldots \equiv x_p - x_{p-1}\pmod{p}.$
    If $x_p = 1$, then $h = (0, p-1, p-2, \ldots, 1)$.
    In fact, $\C_H(a) = \gen{(0, p-1, p-2, \ldots, 1)}$. 
    Therefore $|\C_H(a)|=p$.
    
    Let $h_1 a_1 b_1 \in \C_G(a)$, where $h_1\in H$, $a_1\in A$ and $b_1\in B$.
    We have $h_1 a_1 b_1=(h_1 a_1 b_1)^a = h_1^a a_1 b_1^a = h_1^a (a_1 a^{-1}a^{b_1^{-1}}) b_1$.
    It follows that $h_1\in C_H(a)$ and $b_1=1$.
    Hence $\C_G(a) = \C_H(a) A$.
    Thus, $|\C_G(a)|= p^2$ and so $\Ind_G(a) = p^{p-2}q$.

(2) $|\C_H(b)|=p^2$ and $\Ind_G(b)=p^{p-2}$.

    It is easy to verify that $\C_G(b)=\gen{k_{m_1}\ldots k_{m_q}N, k_{n_1}\ldots k_{n_q}N}$.
    Therefore, $\C_H(b)=p^2$.
    Moreover, $\C_H(a)\cap \C_H(b)=1$.

    If $h_1 a_1 b_1\in \C_G(b)$, then
    $h_1 a_1 b_1 = (h_1 a_1 b_1)^b = h_1^b a_1^b b_1$.
    It follows that $h_1\in \C_H(b)$ and $a_1 =1$.
    Therefore $\C_G(b)= \C_H(b)B$.
    We have $|\C_G(b)|=p^2q$ and so $\Ind_G(b)=p^{p-2}$.

(3) $\Ind_G(ab)=p^{p-2}$.

    By Sylow's theorems, $AB$ has $p$ Sylow $q$-subgroups.
    Since that $p(q-1)+p = pq = |AB|$, 
    every element in $AB-A$ has order $q$.
    Hence $ab$ must be contained in some conjugate of $B$.
    Thus,  $\Ind_G(ab) = \Ind_G(b) = p^{p-2}$.

(4) $\{\Ind_G(h)\mid h\in H\}= \{p,q,pq\}$.

    It is clear that $H\leq \C_G(h)$.
    If $h\in \C_G(a)$, then $\C_G(h)=HA$ and $\Ind_G(h)=q$.
    If $h\in \C_G(b)$ or $\C_G(ab)$, $\C_G(h)=HB$ or $H\gen{ab}$ and $\Ind_G(h)=p$.
    The number of such $h$ in all the cases above 
    is at most $|\C_H(a)| + p|\C_H(b)| = p^3 + p$. 
    Here $p$ must be greater than or equal to $5$, 
    so $p^3 + p <p^{p-1}=|H|$.
    Hence there exists $h\in H$ such that $\C_G(H)=H$.
    For such $h$, $\Ind_G(h)=pq$.

(5) $\Ind_G(ha)=p^{p-2}q$.

    Let $h_1 a_1 b_1 \in \C_G(ha)$, where $h_1\in H$, $a_1\in A$ and $b_1\in B$.
    We have $ha = (ha)^{h_1 a_1 b_1} = (h h_1^{-1} h_1^{a^{-1}})^{a_1 b_1} a^{b_1}$.
    Hence $b_1=1$ and so $\C_G(ha)= \C_{HA}(ha)$.
    We have $(ha)^p=(h h^{a^{-1}} h^{a^{-2}}\ldots h^{a^{1-p}}) a^p
    =h h^{a^{-1}} h^{a^{-2}}\ldots h^{a^{1-p}}$.
    If $h=(x_1, x_2, \ldots, x_p)$, then
    $h h^{a^{-1}} h^{a^{-2}}\ldots h^{a^{1-p}}
    =(x_1+\ldots + x_p, \ldots, x_1+\ldots + x_p) = 1$.
    Hence $ha$ is an element of order $p$.
    Since $\gen{ha}\cap H=1$, 
    We have $HA=H\rtimes \gen{ha}$.
    Therefore $\C_G(ha)=\C_H(ha)\gen{ha}=\C_H(a)\gen{ha}$.
    Thus $|\C_G(ha)|=p^2$ and so $\Ind_G(ha)=p^{p-2}q$.

(6) $\Ind_G(hab)=\Ind_G(hb)=p^{p-2}$.

    By (3), $ab$ and $b$ are conjugate, 
    so $hab$ must be conjugate to $h'b$ 
    where $h'$ is some element in $H$. 
    Thus, we only need to consider $\Ind_G(hb)$.
    Let $h_1 a_1 b_1 \in \C_G(hb)$ where $h_1\in H, a_1\in A, b_1\in B$.
    We have
    $
        hb
        = (hb)^{h_1 a_1 b_1}=(h h_1^{-1} h_1^{b^{-1}} b)^{a_1 b_1}
        = (h h_1^{-1} h_1^{b^{-1}})^{a_1 b_1} (a_1^{-1} a_1^{b^{-1}})^{b_1} b
    $.
    Hence $a_1=1$ and $\C_G(hb)=\C_{HB}(hb)$.
    It follows that $\Ind_G(hb) = \Ind_{HB}(hb)\times p$.
    By Lemma \ref{L2}, $\Ind_{HB}(hb)=|H:\C_H(b)|=p^{p-3}$.
    Therefore, $\Ind_G(hb)=p^{p-2}$.
    
    From (1)--(6), all nontrivial elements of $G$ have been considered, 
    so $\N(G)=\{1, p, p^{p-2}\}\times \{1, q\}$ 
    and $\Z(G)=1$.
    The theorem is proved.

%***************************************************************************

%***************************************************************************

\begin{thebibliography}{s2}

\bibitem{Baer1953} R. Baer, 
    Group elements of prime power index, 
    \textit{Transactions of the American Mathematical Society}, \textbf{75}(1) (1953), 20--47.

\bibitem{Bastan2018} R. Bastan, C. Akin, 
    Notes on Sophie Germain Primes, 
    \textit{Proceedings of International Conference on Mathematics}, \textbf{10} (2018), 18--21.

\bibitem{Camina2006} A. R. Camina, R. D. Camina, 
    Recognising nilpotent groups, 
    \textit{Journal of Algebra}, \textbf{300}(1) (2006), 16--24.

\bibitem{Camina2011} A. R. Camina, R. D. Camina, 
    The influence of conjugacy class sizes on the structure of finite groups: a survey, 
    \textit{Asian-European Journal of Mathematics}, \textbf{4}(4) (2011), 559--588.

\bibitem{Casolo2012} C. Casolo, E. M. Tombari, 
    Conjugacy class sizes of certain direct products, 
    \textit{Bulletin of the Australian Mathematical Society}, \textbf{85}(2) (2012), 217--231.

\bibitem{Cossey2000} J. Cossey, T. Hawkes, 
    Sets of $p$-powers as conjugacy class sizes, 
    \textit{Proceedings of the American Mathematical Society}, \textbf{128}(1) (2000), 49--51.

\bibitem{GAP} The GAP Group, 
    \textit{GAP -- Groups, Algorithms, and Programming}, Version 4.14.0 (2024). 
    (https://www.gap-system.org)

\bibitem{Gorshkov2023} I. B. Gorshkov, 
    Structure of finite groups with restrictions on the set of conjugacy classes sizes, 
    \textit{Communications in Mathematics}, \textbf{32}(1) (2023), 63--71.

\bibitem{Gorshkov2024} I. B. Gorshkov, C. Shao, T. M. Mudziiri Shumba, 
    Description of direct products by sizes of conjugacy classes, 
    \textit{Communications in Algebra}, \textbf{Online First} (2024), 1--7.

\bibitem{Ishikawa2002} K. Ishikawa, 
    On finite $p$-groups which have only two conjugacy lengths, 
    \textit{Israel Journal of Mathematics}, \textbf{129} (2002), 119--123.

\bibitem{Ito1953} N. It\^o, 
    On finite groups with given conjugate types \Romannum{1}, 
    \textit{Nagoya Mathematical Journal}, \textbf{6} (1953), 17--28.

\end{thebibliography}
\end{document}